\documentclass{amsart}

\theoremstyle{definition}
\newtheorem{theorem}{Theorem}[section]
\newtheorem{lemma}[theorem]{Lemma}
\newtheorem{corollary}[theorem]{Corollary}
\newtheorem{conjecture}[theorem]{Conjecture}
\newtheorem{note}[theorem]{Note}
\newtheorem{prop}[theorem]{Proposition}
\newtheorem{fact}[theorem]{Experimental Fact}

\newtheorem{definition}[theorem]{Definition}
\newtheorem{example}[theorem]{Example}

\theoremstyle{remark}

\numberwithin{equation}{section}

\reversemarginpar

\newcommand{\ba}{\begin{eqnarray}}
\newcommand{\ea}{\end{eqnarray}}

\begin{document}

\title[A sequence related to Narayana polynomials]
{A probabilistic interpretation of a sequence related to Narayana 
polynomials}

\author[T. Amdeberhan]{Tewodros Amdeberhan}
\address{Department of Mathematics,
Tulane University, New Orleans, LA 70118}
\email{tambeder@tulane.edu}

\author[V. Moll]{Victor H. Moll}
\address{Department of Mathematics,
Tulane University, New Orleans, LA 70118}
\email{vhm@math.tulane.edu}

\author[C. Vignat]{Christophe Vignat}
\address{Information Theory Laboratory, E.P.F.L., 1015 Lausanne, Switzerland}
\email{christophe.vignat@epfl.ch}


\date{\today}


\begin{abstract}
A sequence of coefficients appearing in a recurrence for the Narayana 
polynomials is generalized. The coefficients are given a probabilistic 
interpretation in terms of beta distributed random variables. The 
recurrence established by M. Lasalle is then obtained from a classical 
convolution identity. Some arithmetical properties of the generalized 
coefficients are also established.
\end{abstract}

\maketitle


\vskip 20pt 

\section{Introduction} 
\label{S:intro} 

The Narayana polynomials 
\begin{equation}
\mathcal{N}_{r}(z) = \sum_{k=1}^{r} N(r,k) z^{k-1}
\label{nara-poly-def}
\end{equation}
\noindent
with the Narayana numbers $N(r,k)$ given by 
\begin{equation}
N(r,k) = \frac{1}{r} \binom{r}{k-1} \binom{r}{k}
\label{nara-numb-def}
\end{equation}
\noindent
have a large number of combinatorial properties. In a recent paper, 
M. Lasalle \cite{lasallem-2012a} established the recurrence 
\begin{equation}
(z+1)\mathcal{N}_{r}(z) - \mathcal{N}_{r+1}(z) 
= \sum_{n \geq 1} (-z)^{n} \binom{r-1}{2n-1}
A_{n} \mathcal{N}_{r-2n+1}(z).
\label{def0-A}
\end{equation}

The numbers $A_{n}$ satisfies the recurrence 
\begin{equation}
(-1)^{n-1}A_{n} = C_{n} + 
\sum_{j=1}^{n-1} (-1)^{j} \binom{2n-1}{2j-1} A_{j}C_{n-j},
\label{def-An}
\end{equation}
\noindent
with $A_{1} = 1$ and $C_{n} = \frac{1}{n+1} \binom{2n}{n}$ the Catalan
number. This recurrence is taken here as being the definition 
of $A_{n}$. The first few values are 
\begin{equation}
A_{1} = 1, \, A_{2} = 1, \, A_{3} = 5, \, A_{4} = 56, \, A_{5} = 1092, 
\, A_{6} = 32670.
\end{equation}

Lasalle \cite{lasallem-2012a} shows that 
$\{ A_{n}: \, n \in \mathbb{N} \}$ is an increasing sequence of 
positive integers.  In 
the process of establishing the positivity of this sequence, he
contacted D. Zeilberger, who suggested the study of the related sequence
\begin{equation}
a_{n} = \frac{2 A_{n}}{C_{n}},
\end{equation}
\noindent
with first few values
\begin{equation}
a_{1} = 2, \, a_{2} = 1, \, a_{3} = 2, \, a_{4} = 8, \, a_{5} = 52, 
\, a_{6} = 495, \, a_{7} = 6470.
\end{equation}

The recurrence \eqref{def-An} yields 
\begin{equation}
(-1)^{n-1} a_{n} = 2 + \sum_{j=1}^{n-1} (-1)^{j} \binom{n-1}{j-1} 
\binom{n+1}{j+1} \frac{a_{j}}{n-j+1}.
\label{def-an}
\end{equation}
\noindent
This may be expressed in terms of the numbers 
\begin{equation}
\sigma_{n,r}:= \frac{2}{n} \binom{n}{r-1} \binom{n+1}{r+1}
\label{sigma-def}
\end{equation}
\noindent
that appear as entry $A108838$ in $OEIS$ and count Dyck paths by the number of 
long interior inclines. The fact that $\sigma_{n,r}$ is an integer also follows 
from 
\begin{equation}
\sigma_{n,r} = 
\binom{n-1}{r-1} \binom{n+1}{r} - \binom{n-1}{r-2} \binom{n+1}{r+1}.
\end{equation}
\noindent
The relation \eqref{def-an} can also be written as 
\begin{equation}
a_{n} = (-1)^{n-1} \left[ 2 + \frac{1}{2} \sum_{j=1}^{n-1} (-1)^{j} 
\sigma_{n,j} a_{j} \right].
\label{relation-1}
\end{equation}

The original approach by M. Lasalle \cite{lasallem-2012a} is to establish 
the relation 
\begin{equation}
(z+1) \mathcal{N}_{r}(z) - \mathcal{N}_{r+1}(z) = 
\sum_{n \geq 1} (-z)^{n} \binom{r-1}{2n-1} 
A_{n}(r) \mathcal{N}_{r-2n+1}(z)
\label{recu-lasalle}
\end{equation}
\noindent
for some coefficient $A_{n}(r)$. The expression 
\begin{equation}
\mathcal{N}_{r}(z) = \sum_{m \geq 0} z^{m} (z+1)^{r-2m-1} \binom{r-1}{2m} C_{m}
\label{nara-cata0}
\end{equation}
\noindent
given in \cite{coker-2003a}, is then employed 
to show that $A_{n}(r)$ is independent of $r$. This is the definition of 
$A_{n}$ given in \cite{lasallem-2012a}. Lasalle mentions in passing that 
``J. Novak observed, as empirical evidence, that the integers $(-1)^{n-1}A_{n}$
are precisely the (classical) cumulants of a standard semicircular 
random variable''. 

The goal of this paper is to revisit Lasalle's results, provide probabilistic
interpretation of the numbers $A_{n}$ and to consider
Zeilberger's suggestion.

\medskip

The probabilistic interpretation of the numbers $A_{n}$ starts with the 
semicircular distribution
\begin{equation}
f_{1}(x) = \begin{cases}
        \frac{2}{\pi} \sqrt{1-x^{2}} & \quad \text{ if } -1 \leq x \leq 1 \\
       0  & \quad \text{ otherwise}.
     \end{cases}
\end{equation}
\noindent
Let $X$ be a  random variable with distribution $f_{1}$. Then $X_{*} = 2X$
satisfies 
\begin{equation}
\mathbb{E} \left[ X_{*}^{r} \right] = 
\begin{cases} 
0 & \quad \text{ if } r \text{ is odd} \\
C_{m}  & \quad \text{ if } r \text{ is even, with } r = 2m,
\end{cases}
\end{equation}
\noindent
where $C_{n} = \frac{1}{m+1} \binom{2m}{m}$ are the Catalan numbers. The 
moment generating function 
\begin{equation}
\varphi(t) = \sum_{n=0}^{\infty} \mathbb{E} \left[ X^{n} \right] \frac{t^{n}}
{n!} 
\end{equation}
\noindent
is expressed in terms of the modified Bessel function of the first kind 
$I_{\alpha}(x)$ and the cumulant generating function 
\begin{equation}
\psi(t) = \log \varphi(t) = \sum_{n=1}^{\infty} \kappa_{1}(n) \frac{t^{n}}{n!}
\end{equation}
\noindent
has coefficients $\kappa_{1}(n)$, known as the cumulants of $X$. The 
identity 
\begin{equation}
A_{n} = (-1)^{n+1} \kappa_{1}(2n) 2^{2n},
\end{equation}
\noindent
is established here. Lasalle's 
recurrence \eqref{def-An} now follows from the convolution 
identity 
\begin{equation}
\kappa(n) = \mathbb{E} \left[ X^{n} \right] - 
\sum_{j=1}^{n-1} \binom{n-1}{j-1} \kappa(j) \mathbb{E} \left[ X^{n-j} \right]
\label{convolution}
\end{equation}
\noindent
that holds for any pair of moments and cumulants sequences \cite{smith-1995a}.
The coefficient $a_{n}$ suggested by D. Zeilberger now takes the form
\begin{equation}
a_{n} = \frac{2 (-1)^{n+1} \kappa_{1}(2n)}{\mathbb{E} \left[X_{*}^{2n} 
\right]}.
\label{an-def11}
\end{equation}

In this paper, these 
notions are extended to the case of random variables distributed 
according to the symmetric beta distribution 
\begin{equation}
\label{semicircular0}
f_{\mu}(x) = 
\frac{1}{B(\mu + \tfrac{1}{2}, \tfrac{1}{2})} 
(1-x^{2})^{\mu - 1/2}, \quad \text{ for } |x| \leq 1, \, \mu > 
- \tfrac{1}{2} 
\end{equation}
\noindent 
and $0$ otherwise. The semi-circular distribution is the particular 
case $\mu=1$. Here 
$B(a,b)$ is the classical beta function defined by the integral
\begin{equation}
B(a,b) = \int_{0}^{1} t^{a-1}(1-t)^{b-1} \, dt, \quad \text{ for } 
a, \, b  > 0.
\end{equation}
\noindent
These ideas lead to introduce a generalization of the Narayana polynomials
and these are expressed in terms of the classical Gegenbauer polynomials 
$C_{n}^{\mu + \tfrac{1}{2}}$. The coefficients $a_{n}$ are also 
generalized to a family of 
numbers $\{ a_{n}(\mu) \}$ with parameter $\mu$. The 
special cases $\mu=0$ and $\mu = \pm \tfrac{1}{2}$ are discussed in 
detail.  

\medskip

Section \ref{S:positivity} produces a recurrence for $\{ a_{n} \}$ 
from which the fact that $a_{n}$ is incresaing and positive are established. 
The recurrence comes from a relation between $\{ a_{n} \}$ and the Bessel 
function $I_{\alpha}(x)$. Section \ref{S:determinants} gives an expression 
for $\{ a_{n} \}$ in terms of a determinant of an upper Hessenberg matrix. 
The standard procedure to evaluate these determinants gives the 
original recurrence defining $\{ a_{n} \}$.  Section \ref{S:prob} introduces 
the probabilistic interpretation of the numbers $\{ a_{n} \}$. The cumulants 
of the associated random variable are expressed in terms of the Bessel zeta
function. Section \ref{S:narayana} presents the Narayana polynomials as 
expected values of a simple function of a semicircular random variable. These 
polynomials are generalized in Section \ref{S:generalized} and they are 
expressed in terms of Gegenbauer polynomials. The 
corresponding extension of
$\{ a_{n} \}$ are presented in Section \ref{S:special}.  The paper concludes
with some arithmetical 
properties of $\{ a_{n} \}$ and its generalization corresponding to the 
parameter $\mu = 0$. These are described in Section \ref{S:arithmetic}.

\section{The sequence $\{ a_{n} \}$ is positive and increasing} 
\label{S:positivity} 

In this section a direct proof of the positivity of the 
numbers $a_{n}$ defined in \eqref{def-an} is provided. Naturally 
this implies $A_{n} \geq 0$. 
The analysis employs the \textit{modified
Bessel function of the first kind} 
\begin{equation}
\label{bessel-mod}
I_{\alpha}(z) := \sum_{j=0}^{\infty} \frac{1}{j! \, (j+ \alpha)!} \, 
\left( \frac{z}{2} \right)^{2j+ \alpha}.
\end{equation}
\noindent
Formulas for this function appear in \cite{gradshteyn-2007a}.

\begin{lemma}
\label{lemma-bessel1}
The numbers $a_{n}$ satisfy
\begin{equation}
\sum_{j=1}^{\infty} \frac{(-1)^{j-1} a_{j}}{(j+1)!} \frac{x^{j-1}}{(j-1)!} 
= \frac{2}{\sqrt{x}} \frac{I_{2}(2 \sqrt{x})}{I_{1}(2 \sqrt{x})}.
\end{equation}
\end{lemma}
\begin{proof}
The statement is equivalent to 
\begin{equation}
\sqrt{x} I_{1}(2 \sqrt{x}) \times 
\sum_{j=1}^{\infty} \frac{(-1)^{j-1} a_{j}}{(j+1)!} \frac{x^{j-1}}{(j-1)!} 
 = 2 I_{2}(2 \sqrt{x}).
\end{equation}
\noindent
This is established by comparing coefficients of $x^{n}$ on both sides 
and using \eqref{def-an}.
\end{proof}

Now change $x$ to $x^{2}$ in Lemma \ref{lemma-bessel1} to write 
\begin{equation}
\sum_{j=1}^{\infty} \frac{(-1)^{j-1} a_{j}}{(j+1)!} \frac{x^{2j-2}}{(j-1)!} 
=   \frac{2}{x} \frac{I_{2}(2x)}{I_{1}(2x)}.
 \label{nice-0}
\end{equation}
\noindent
The classical relations 
\begin{equation}
\frac{d}{dz} \left( z^{-m} I_{m}(z) \right) = z^{-m} I_{m+1}(z), \text{ and }
\frac{d}{dz} \left( z^{m+1} I_{m+1}(z) \right) = z^{m+1} I_{m}(z)
\label{classical}
\end{equation}
\noindent
give 
\begin{equation}
I_{1}'(z) = I_{2}(z) + \frac{1}{z} I_{1}(z).
\end{equation}
\noindent
Therefore \eqref{nice-0} may be written as 
\begin{equation}
\sum_{j=1}^{\infty} \frac{(-1)^{j-1} a_{j}}{(j+1)!} \frac{x^{2j-2}}{(j-1)!} 
= \frac{1}{x} \frac{d}{dx} \log \left( \frac{I_{1}(2x)}{2x} \right).
\label{nice-1}
\end{equation}

The relations \eqref{classical} also produce 
\begin{equation}
\frac{d}{dz} \left( \frac{z^{m+1} I_{m+1}(z)}{z^{-m} I_{m}(z)} \right) 
= z^{2m+1} \frac{I_{m}^{2}(z) - I_{m+1}^{2}(z)}{I_{m}^{2}(z)}.
\end{equation}
\noindent
In particular,
\begin{equation}
\frac{d}{dz} \left( \frac{z^{2} I_{2}(z)}{z^{-1} I_{1}(z)} \right) 
= z^{3}  - z^{3} \frac{I_{2}^{2}(z)}{I_{1}^{2}(z)}.
\end{equation}
\noindent
Replacing this relation in \eqref{nice-1} gives the recurrence stated next.

\begin{prop}
\label{recu-fora}
The numbers $a_{n}$ satisfy the recurrence 
\begin{equation}
2na_{n} = \sum_{k=1}^{n-1} \binom{n}{k-1} \binom{n}{k+1} a_{k} a_{n-k},
\quad \text{ for } n \geq 2,
\label{rec-11}
\end{equation}
\noindent
with initial condition $a_{1} =1$.
\end{prop}

\begin{corollary}
\label{nonneg}
The numbers $a_{n}$ are nonnegative.
\end{corollary}

\begin{prop}
\label{prop-rec}
The numbers $a_{n}$ satisfy
\begin{equation}
4a_{n} = \sum_{k=1}^{n-1} \binom{n-1}{k-1} \binom{n-1}{k} a_{k}a_{n-k} 
- \sum_{k=2}^{n-2} \binom{n-1}{k-2} \binom{n-1}{k+1} a_{k}a_{n-k}.
\end{equation}
\end{prop}
\begin{proof}
This follows from \eqref{rec-11} and the identity 
\begin{equation*}
\binom{n}{k-1} \binom{n}{k+1} = \frac{n}{2} 
\left[ \binom{n-1}{k-1} \binom{n-1}{k} - \binom{n-1}{k-2} \binom{n-1}{k+1} 
\right].
\end{equation*}
\end{proof}

\begin{corollary}
The numbers $a_{n}$ are nonnegative integers. Moreover $a_{n}$ is 
even if $n$ is odd.
\end{corollary}
\begin{proof}
Corollary \ref{nonneg} shows $a_{n} > 0$. It remains to show $a_{n} \in 
\mathbb{Z}$ and to verify the parity statement.  This is achieved by 
simultaneous induction on $n$. 

\smallskip

Assume first $n = 2m+1$ is odd. Then \eqref{sigma-def} shows that 
$\tfrac{1}{2}\sigma_{n,r}  \in 
\mathbb{Z}$ and \eqref{relation-1}, written as 
\begin{equation}
a_{n} = (-1)^{n-1} \left[ 2 + \sum_{r=1}^{n-1} 
\frac{\sigma_{n,r}}{2} a_{r} \right],
\label{relation-1a}
\end{equation}
\noindent
proves that $a_{n} \in \mathbb{Z}$. Now write \eqref{rec-11} as 
\begin{equation}
2(2m+1) a_{2m+1} = 2 \sum_{k=1}^{m} \binom{2m+1}{k-1} \binom{2m+1}{k+1}
a_{k} a_{2m+1-k}
\end{equation}
\noindent 
and observe that either $k$ or $2m+1-k$ is odd. The induction hypothesis
shows that either $a_{k}$ or $a_{2m+1-k}$ is even. This shows 
$a_{2m+1}$ is even. 

\smallskip 

Now consider the case $n = 2m$ even. If $r$ is odd, then $a_{r}$ is even; if 
$r$ is even then $r-1$ is odd and $\tfrac{1}{2}\sigma_{n,r} \in \mathbb{Z}$
in view of the identity
\begin{equation}
\sigma_{n,r} = \frac{2}{r-1} \binom{n-1}{r-2} \binom{n+1}{r+1}.
\end{equation}
\noindent
The result follows again from \eqref{relation-1a}.
\end{proof}

\begin{corollary}
The numbers $A_{n}$ are nonnegative integers.
\end{corollary}

The recurrence in Proposition \ref{recu-fora} is now employed to 
prove that $\{ a_{n} \}$ is an increasing sequence. The first few values 
are $2, \, 1, \, 2, \, 8, \, 52$.

\begin{theorem}
For $n \geq 3$, the inequality $a_{n} >  a_{n-1}$ holds.
\end{theorem}
\begin{proof}
Take the terms $k=1$ and $k=n-1$ in the sum appearing in the 
recurrence in Proposition \eqref{recu-fora} and use $a_{n} > 0$ to obtain 
\begin{equation}
a_{n} \geq \frac{1}{2n} \left[ \binom{n}{0} \binom{n}{2} a_{1}a_{n-1}+ 
\binom{n}{n-2} \binom{n}{2} a_{n-1}a_{1} \right].
\end{equation}
\noindent
Since $a_{1}=2$ the previous inequality yields
\begin{equation}
a_{n} \geq (n-1)a_{n-1}.
\end{equation}
\noindent
Hence, for $n \geq 3$, this gives $a_{n} - a_{n-1} \geq (n-2)a_{n-1} > 0$.
\end{proof}

\medskip

\section{An expression in forms of determinants} 
\label{S:determinants} 

The recursion relation \eqref{def-an} expressed in the form
\begin{equation}
\sum_{j=1}^{m} (-1)^{j-1} \binom{m}{j-1} \binom{m+1}{j+1}a_{j} = 2m
\end{equation}
\noindent
is now employed to produce a system of equations for the numbers 
$a_{n}$ by varying $m$ through $1, \, 2, \, 3, \cdots, n$.  The coefficient 
matrix has determinant $(-1)^{\binom{n}{2}} n!$ and Cram\'{e}r's rule gives 

\smallskip

\begin{equation}
a_{n} = \frac{(-1)^{n-1}}{n!} \det 
\begin{pmatrix} 
\binom{1}{1-1} \binom{1+1}{1+1} & 0 & 0 & \cdots & 0  & 2 \\
\binom{2}{1-1} \binom{2+1}{1+1} & \binom{2}{2-1} \binom{2+1}{2+1} 
 & 0 & \cdots & 0  & 4 \\
\binom{3}{1-1} \binom{3+1}{1+1} & \binom{3}{2-1} \binom{3+1}{2+1} 
 & \binom{3}{3-1} \binom{3+1}{3+1} & \cdots & 0  & 6 \\
\cdots &  \cdots &  \cdots & \cdots & \cdots & \cdots \\
\cdots &  \cdots &  \cdots & \cdots & \cdots & \cdots \\
\binom{n}{1-1} \binom{n+1}{1+1} & \binom{n}{2-1} \binom{n+1}{2+1} & 
\binom{n}{3-1} \binom{n+1}{3+1} & \cdots & \binom{n}{n-2} \binom{n+1}{n} 
& 2n 
\end{pmatrix}
\end{equation}

\noindent
The power of $-1$ is eliminated by permuting the columns to produce the 
matrix
\begin{equation}
B_{n} = \begin{pmatrix}
2 & \binom{1}{1-1} \binom{1+1}{1+1} & 0 & 0 & 0  \\
4 & \binom{2}{1-1} \binom{2+1}{1+1} & \binom{2}{2-1} \binom{2+1}{2+1} & 0 
& \cdots \\
6 & \binom{3}{1-1} \binom{3+1}{1+1} & \binom{3}{2-1} \binom{3+1}{2+1} & 0 
& \cdots \\
\cdots & \cdots & \cdots &  \cdots & \cdots \\
\cdots & \cdots & \cdots &  \cdots & \cdots \\
2n & \binom{n}{1-1} \binom{n+1}{1+1} & \binom{n}{2-1} \binom{n+1}{2+1} 
& \binom{n}{3-1} \binom{n+1}{3+1}  \cdots & \binom{n}{n-2} \binom{n+1}{n}
\end{pmatrix}.
\label{matrix-b}
\end{equation}
\noindent
The representation of $a_{n}$ in terms of determinants is given in the 
next result.

\begin{prop}
The number $a_{n}$ is given by 
\begin{equation}
a_{n} = \frac{\det B_{n} }{n!}
\label{new-def-an}
\end{equation}
\noindent
where $B_{n}$ is the matrix in \eqref{matrix-b}.
\end{prop}

Recall that an \textit{upper Hessenberg matrix} is one of the form 
\begin{equation}
H_{n} = \begin{pmatrix}
\beta_{1,1} & \beta_{1,2} & 0 & 0 & 0 & \cdots & \cdots & 0 & 0 \\
\beta_{2,1} & \beta_{2,2} & \beta_{2,3} & 0 & 0 & \cdots & \cdots & 0 & 0 \\
\beta_{3,1} & \beta_{3,2} & \beta_{3,3} & \beta_{3,4} & 0 & \cdots & \cdots & 0 & 0 \\
\cdots & \cdots & \cdots & \cdots  & 
\cdots & \cdots & \cdots &  \cdots & \cdots \\
\cdots & \cdots & \cdots & \cdots  & 
\cdots & \cdots & \cdots &  \cdots & \cdots \\
\beta_{n,1} & \beta_{n,2} & \beta_{n,3} & \beta_{n,4} & 
\cdots & \cdots & \cdots &  \beta_{n,n-1} & \beta_{n,n} 
\end{pmatrix}
\end{equation}
\noindent
The matrix $B$ is of this form with 
\begin{equation}
\beta_{i,j} = 
\begin{cases}
2i  & \quad \text{ if } 1 \leq i \leq n  \quad \text{ and } j = 1 \\
\binom{i}{j-2} \binom{i+1}{j}  & \quad \text{ if } j-1 \leq i \leq n 
 \text{ and } j > 1.
\end{cases}
\end{equation}

It turns out that the recurrence \eqref{def-an} used to define the numbers 
$a_{n}$ can be recovered if one employs \eqref{new-def-an}. 

\begin{prop}
Define $\alpha_{n}$ by 
\begin{equation}
\alpha_{n} = \frac{\det B_{n} }{n!}
\end{equation}
\noindent
where $B$ is the matrix \eqref{matrix-b}. Then $\{ \alpha_{n} \}$ 
satisfies the recursion 
\begin{equation}
(-1)^{n-1} \alpha_{n} = 2 + \sum_{j=1}^{n-1} (-1)^{j} \binom{n-1}{j-1} 
\binom{n+1}{j+1} \frac{\alpha_{j}}{n-j+1}
\label{recu-alpha}
\end{equation}
\noindent
and the initial condition $\alpha_{1} = 1$. Therefore $\alpha_{n} = a_{n}$.
\end{prop}
\begin{proof}
For convenience define $\det H_{0} = 1$. The determinant of a Hessenberg 
matrix satisfies the recurrence 
\begin{equation}
\det H_{n} = \sum_{r=1}^{n} (-1)^{n-r} \beta_{n,r} \det H_{r-1} 
\prod_{i=r}^{n-1} \beta_{i,i+1}.
\label{det-hess}
\end{equation}
\noindent
A direct application of \eqref{det-hess} yields
\begin{eqnarray*}
\alpha_{n} & = & \frac{1}{n!} 
\left\{ (-1)^{n-1} (2n) (n-1)! + 
\sum_{r=2}^{n} (-1)^{n-r} \binom{n}{r-2} \binom{n+1}{r} 
\det B_{r-1} \prod_{i=r}^{n-1} i \right\} \\
& = & 2 (-1)^{n-1} + \frac{1}{n!} \sum_{r=2}^{n} (-1)^{n-r} 
\binom{n}{r-2} \binom{n+1}{r} \alpha_{r-1} \, (n-1)! \\
& = & 2 (-1)^{n-1} + \sum_{r=2}^{n} (-1)^{n-r}  \frac{1}{n}
\binom{n}{r-2} \binom{n+1}{r} \alpha_{r-1}  \\
& = & 2 (-1)^{n-1} + \sum_{r=2}^{n} (-1)^{n-r} 
\binom{n}{r-2} \binom{n+1}{r} \frac{\alpha_{r-1}}{n-r+2}   \\
& = & 2 (-1)^{n-1} + (-1)^{n-1} \sum_{r=1}^{n} (-1)^{j} 
\binom{n-1}{j-1} \binom{n+1}{j+1} \frac{\alpha_{j}}{n-j+1}.
\end{eqnarray*}
\noindent 
This is \eqref{recu-alpha}.
\end{proof}

\begin{corollary}
The modified Bessel function of the first kind admits a determinant 
expression
\begin{equation}
I_{1}(x) = x \, \exp \left( \sum_{j=1}^{\infty} 
\frac{(-1)^{j-1} \det B_{j}}{(j+1)! \, j!^{2}} \left( \frac{x}{2} \right)^{2j}
\right).
\end{equation}
\end{corollary}
\begin{proof}
This follows by integrating the identity
\begin{equation}
\frac{2 I_{2}(2x)}{x \, I_{1}(2x)} = \frac{1}{x} \frac{d}{dx} 
\log \frac{I_{1}(2x)}{2x}.
\end{equation}
\end{proof}

\section{The probabilistic background: conjugate random variables} 
\label{S:prob} 

This section provides the probabilistic tools required for an 
interpretation of the sequence $A_{n}$
defined in \eqref{def-An}. The specific connections are given in Section 
\ref{S:narayana}.

\smallskip

Consider a random variable $X$ with the \textit{symmetric beta distribution}
given in \eqref{semicircular0}.  The 
moments of the symmetric beta distribution, given by 
\begin{equation}
\mathbb{E}\left[X^{n} \right] = \frac{1}{B(\mu + \tfrac{1}{2}, \tfrac{1}{2})} 
\int_{-1}^{1} x^{n} (1-x^{2})^{\mu-1/2} \, dx,
\end{equation}
\noindent
vanish for $n$ odd and for $n = 2m$ they are
\begin{equation}
\mathbb{E}\left[X^{2m} \right] = 
\frac{\Gamma(\mu+1)}{\Gamma(\mu+1+m)} \frac{(2m)!}{2^{2m} \, m!}.
\end{equation}
\noindent
Therefore the moment generating function is 
\begin{equation}
\label{mom-series}
\varphi_{\mu}(t) = \mathbb{E} \left[e^{t X}\right] = 
\sum_{n=0}^{\infty} \mathbb{E} \left[X^{n} \right] \frac{t^{n}}{n!} = 
\Gamma(\mu+1) \sum_{m=0}^{\infty} \frac{t^{2m}}
{2^{2m} \, m! \, \Gamma( \mu + m + 1)}.
\end{equation}
\noindent 
The next proposition summarizes properties of $\varphi_{\mu}(t)$. The first 
one is to recognize the series in \eqref{mom-series} from 
\eqref{bessel-mod}. The zeros $\{j_{\mu,k} \}$ of the Bessel function of the 
first kind 
\begin{equation}
J_{\alpha}(x) = \sum_{j=0}^{\infty} \frac{(-1)^{m}}{m! \, \Gamma(m + 
\alpha + 1)} \left( \frac{x}{2} \right)^{2m+\alpha}
\end{equation}
\noindent
appear in the factorization of $\varphi_{\mu}$ in view of the relation
$I_{\mu}(z) = e^{-\pi i \mu/2} J_{\mu}(iz)$.

\begin{prop}
The 
moment generating function $\varphi_{\mu}(t)$ of a random variable $X \sim 
f_{\mu}$ is given by 
\begin{equation}
\varphi_{\mu}(t) = \Gamma(\mu+1) \left( \frac{2}{t} \right)^{\mu} 
I_{\mu}(t).
\end{equation}
\end{prop}

\begin{note}
The 
Catalan numbers $C_{n}$ appear as the even-order moments of $f_{\mu}$ 
when $\mu=1$. More precisely, if $X$ is distributed as $f_{1}$ (written as 
$X \sim f_{1}$), then 
\begin{equation}
\mathbb{E}\left[(2X)^{2n} \right] = C_{n} \text{ and } 
\mathbb{E} \left[(2X)^{2n+1} \right] = 0.
\label{exp-catalan}
\end{equation}
\end{note}

\begin{note}
The moment generating function of $f_{\mu}$ 
admits the Weierstrass product representation
\begin{equation}
\varphi_{\mu}(t) = \prod_{k=1}^{\infty} \left( 1 + \frac{t^{2}}{j_{\mu,k}^{2}} 
\right)
\label{factor}
\end{equation}
\noindent
where $\{ j_{\mu,k} \}$ are the zeros of the Bessel function of the first 
kind $J_{\mu}$. 
\end{note}

\begin{definition}
\label{def-cum}
The \textit{cumulant generating function} is
\begin{eqnarray*}
\psi_{\mu}(t) & = & \log \varphi_{\mu}(t) \\
    & = & \log \left(  \sum_{n=0}^{\infty} 
\mathbb{E} \left[X^{n} \right] \frac{t^{n}}{n!} 
\right).
\end{eqnarray*}
\end{definition}

The product representation of $\varphi_{\mu}(t)$ yields
\begin{eqnarray*}
\log \varphi_{\mu}(t) & = & \sum_{k=1}^{\infty} \log \left( 1 + \frac{t^{2}}
{j_{\mu,k}^{2}} \right) \\
& = & \sum_{k=1}^{\infty} \sum_{n=1}^{\infty} 
\frac{(-1)^{n-1}}{n} \left( \frac{t}{j_{\mu,k}} \right)^{2n} \\
& := & \sum_{n=1}^{\infty} \kappa_{\mu}(n) \frac{t^{n}}{n!}.
\end{eqnarray*}
\noindent
The series converges for $|t| < j_{\mu,1}$. The first Bessel zero satisfies 
$j_{\mu,1} > 0$ for all $\mu \geq 0$. It 
follows that the series has a 
non-zero radius of convergence. 

\smallskip

\begin{note}
The coefficient $\kappa_{\mu}(n)$ is the $n$-th \textit{cumulant} of $X$. 
An expression that links the moments to the cumulants of $X$
is provided by V. P. Leonov and A. N. Shiryaev \cite{leonov-1959a}:
\begin{equation}
\kappa_{\mu}(n) = \sum_{\mathcal{V}} (-1)^{k-1} (k-1)! 
\prod_{i=1}^{k} \mathbb{E}(2X)^{|V_{i}|}
\end{equation}
\noindent
where the sum is over all partitions 
$\mathcal{V} = \{ V_{1}, \, \cdots, V_{k} \}$
of the set $\{ 1, \, 2, \, \ldots, n \}$. 

In the case $\mu = 0$ the moments are Catalan numbers or $0$, in the case 
$\mu=1$ the moments are central binomial coefficients. Therefore, in both 
cases, the 
cumulants $\kappa_{\mu}(n)$ are integers. An expression for the general value 
of $\mu$ involves 
\begin{equation}
\zeta_{\mu}(s) = \sum_{k=1}^{\infty} \frac{1}{j_{\mu,k}^{s}}
\end{equation}
\noindent
the \textit{Bessel zeta function}, sometimes referred as the 
\textit{Rayleigh function}. 
\end{note}

\smallskip

The next result gives an expression for the cumulants of a random 
variable $X$ with a  distribution $f_{\mu}$.  The special case $\mu=1$, 
described in the next section, provides the desired probabilistic 
interpretation of the original sequence $A_{n}$.

\begin{theorem}
Let $X \sim f_{\mu}$. Then 
\begin{equation}
\label{cumu-zeta}
\kappa(n) = 
\begin{cases}
0 & \quad \text{ if } n \text{ is odd},  \\
2 (-1)^{n/2+1} (n-1)! \, \zeta_{\mu}(n) & \quad \text{ if } n \text{ is even}.
\end{cases}
\end{equation}
\end{theorem}
\begin{proof}
Rearranging the expansion in Definition \ref{def-cum} gives
\begin{eqnarray*}
\log \varphi_{\mu}(t) & = & 
\sum_{k=1}^{\infty} \sum_{n=1}^{\infty} \frac{(-1)^{n+1}}{n} 
\left( \frac{t}{j_{\mu,k}} \right)^{2n} \\
& = & \sum_{n=1}^{\infty} \frac{(-1)^{n+1}}{n} t^{2n} 
\sum_{k=1}^{\infty} \frac{1}{j_{\mu,k}^{2n}}.
\end{eqnarray*}
\noindent
Now compare powers of $t$ in this expansion with the definition 
\begin{equation}
\log \varphi_{\mu}(t) = \sum_{n=1}^{\infty} \kappa_{\mu}(n) \frac{t^{n}}{n!}
\end{equation}
\noindent
to obtain the result.
\end{proof}

The next ingredient in the search for an interpretation of the 
sequence $A_{n}$ is the 
notion of conjugate random variables. The properties described below appear 
in \cite{spitzer-1955a}. A complex-valued random variable $Z$ is called a 
\textit{regular random variable} ($rrv$ for short)  if
 $\mathbb{E} |Z|^{n} < \infty$ for all 
$n \in \mathbb{N}$ and 
\begin{equation}
\mathbb{E}\left[ h(Z) \right] = h \left( \mathbb{E} \left[Z \right] \right)
\label{cancel}
\end{equation}
\noindent
for all polynomials $h$. The class of rrv is closed under compositions
with polynomials (if $Z$ is rrv and $P$ is a polynomial, then $P(Z)$ is rrv) 
and it is also closed under addition of independent rrv.  The basic definition
is stated next.

\begin{definition}
Let $X, \, Y$ be real random variables, not necessarily independent. The 
pair $(X,Y)$ is called 
\textit{conjugate random variables} 
if $Z = X + i Y$ is an rrv. The random variable $X$
is called \textit{self-conjugate} if $Y$ has the same distribution as $X$.
\end{definition}

The property of rrv may be expressed in terms of the function 
\begin{equation*}
\Phi(\alpha, \beta) := \mathbb{E} \left[ \text{exp}(i \alpha X + i \beta Y)
\right]
\end{equation*}

The next theorem gives a condition for $Z = X + i Y$ to be an rrv. The 
random variables $X$ and $Y$ are not necessarily independent.

\begin{theorem}
Let $Z = X+iY$ be a complex valued random variable with $\mathbb{E} 
\left[ Z \right] = 0$ and $\mathbb{E} \left[ Z^{n} \right] < \infty$. Then 
$Z$ is an rrv if and only if $\Phi(\alpha, i \alpha) = 1$ for all 
$\alpha \in \mathbb{C}$.
\end{theorem}

This is now reformulated for real and independent random variables.

\begin{theorem}
Let $X, \, Y$ be independent real valued random variables with finite moments.
Define 
\begin{equation*}
\Phi_{X}(\alpha) = \mathbb{E} \left[ e^{i \alpha X} \right] = 
\sum_{n=0}^{\infty} \frac{(i \alpha)^{n}}{n!} 
\mathbb{E} \left[ X^{n} \right] \text{ and }
\Phi_{Y}(\beta) = 
\mathbb{E} \left[ e^{i \alpha Y} \right] = 
\sum_{n=0}^{\infty} \frac{(i \beta)^{n}}{n!} 
\mathbb{E} \left[ Y^{n} \right].
\end{equation*}
\noindent
Then $Z = X + i Y$ is an rrv with mean zero if and only if 
$\Phi_{X}(\alpha) \Phi_{Y}(i \alpha) = 1$.
\end{theorem}

\begin{example}
Let $X$ and $Y$ be independent Gaussian variables with 
zero mean and the same variance. Then $X$ and $Y$ are conjugate since 
\begin{equation*}
\varphi_{X}(t) = \text{exp} \left( \frac{\sigma^{2}}{2}t^{2} \right) 
\text{ and }
\varphi_{iY}(t) = \text{exp} \left( -  \frac{\sigma^{2}}{2}t^{2} \right).
\end{equation*}
\end{example}

\begin{note}
Suppose $Z = X + i Y$ is a rrv with $\mathbb{E}\left[ Z \right] = 0$ 
and $z \in \mathbb{C}$. The condition \eqref{cancel} becomes
\begin{equation}
\mathbb{E}\left[ h(z + X + i Y) \right] = h(z).
\label{cancel2}
\end{equation}
\end{note}

\medskip

Given a sequence of polynomials $\{ Q_{n}(z) \}$ such that $\deg(Q_{n}) = n$
and with leading coefficient $1$, an elementary argument shows that there 
is a unique sequence of coefficients $\alpha_{j,n}$ 
such that the relation
\begin{equation}
Q_{n+1}(z) - z Q_{n}(z) = \sum_{j=0}^{n} \alpha_{j,n} Q_{j}(z)
\label{recu-Q}
\end{equation}
\noindent 
holds. This section discusses this recurrence for the sequence of polynomials 
\begin{equation}
P_{n}(z) := \mathbb{E}(z + X)^{n}
\end{equation}
\noindent
associated to a random variable $X$. The polynomial $P_{n}$ is of 
degree $n$ and has leading coefficient $1$. It is shown that if 
the cumulants of 
odd order vanish, then the even order cumulants provide the 
coefficients $\alpha_{j,n}$ for the recurrence \eqref{recu-Q}.

\begin{theorem}
\label{thm-recurrence0}
Let $X$ be a random variable with cumulants $\kappa(m)$. Assume the 
odd-order cumulants vanish and that $X$ has a
conjugate random variable $Y$. Define the polynomials 
\begin{equation}
P_{n}(z) = \mathbb{E} \left[(z + X)^{n} \right].
\end{equation}
\noindent
Then $P_{n}$ satisfies the recurrence
\begin{equation}
P_{n+1}(z) - zP_{n}(z) = \sum_{m \geq 1} \binom{n}{2m-1} \kappa(2m) 
P_{n-2m+1}(z).
\label{recu-00}
\end{equation}
\end{theorem}
\begin{proof}
Let $X_{1}, \, X_{2}$ independent copies of $X$. Then 
\begin{multline*}
\mathbb{E} \left[ X_{1} \left( (X_{1}+iY_{1}+z + X_{2})^{n} - (z + X_{2})^{n}
\right) \right]  =  \\
=  \sum_{j=0}^{n} \binom{n}{j} \mathbb{E} \left[ X_{1}(X_{1}+iY_{1})^{j} 
(z + X_{2})^{n-j} \right] - 
\mathbb{E} \left[ X_{1} (z + X_{2})^{n} \right].
\end{multline*}
\noindent
This last expression becomes
\begin{equation*}
\sum_{j=1}^{n} \binom{n}{j} \mathbb{E} \left[ X_{1}(X_{1}+iY_{1})^{j} 
(z + X_{2})^{n-j} \right] 
= \sum_{j=1}^{n} \binom{n}{j} \mathbb{E} \left[ X_{1}(X_{1}+iY_{1})^{j}  \right]
\mathbb{E} \left[(z + X_{2})^{n-j} \right].
\end{equation*}

On the other hand
\begin{multline*}
\mathbb{E} \left[ X_{1} \left( (X_{1}+z + X_{2} + iY_{1})^{n} - (z + X_{2})^{n}
\right) \right]  =   \\
\sum_{r=0}^{n} \binom{n}{r} \mathbb{E} \left[ X_{1}(X_{1}+z)^{n-r} \right] 
\mathbb{E} \left[(X_{2}+iY_{1})^{r} \right] - \mathbb{E} 
\left[ X_{1}(z+X_{2})^{n}  \right].
\end{multline*}
\noindent
The cancellation property \eqref{cancellation0} shows that the only surviving 
term in the sum is $r=0$, therefore
\begin{multline*}
\mathbb{E} \left[ X_{1} \left( (X_{1}+z + X_{2} + iY_{1})^{n} - (z + X_{2})^{n}
\right) \right]  =   \\
\mathbb{E} \left[ X_{1}(X_{1}+z)^{n} \right] - \mathbb{E} \left[ X_{1} 
\right] \mathbb{E} \left[ \left( z + X_{2} \right)^{n} \right]
\end{multline*}
\noindent
and $\mathbb{E} \left[X_{1}\right] = 0$ 
since  $\kappa(1) = 0$. This shows the identity
\begin{equation}
\label{nice-010}
\sum_{j=1}^{n} \binom{n}{j} \mathbb{E} \left[ X_{1}(X_{1}+iY_{1})^{j}  \right]
\mathbb{E} \left[(z + X_{2})^{n-j} \right] = 
\mathbb{E} \left[ X_{1}(X_{1}+z)^{n} \right].
\end{equation}

\smallskip 

The cumulants of $X$ satisfy 
\begin{equation}
\kappa(m) = \mathbb{E} X(X+iY)^{m-1}, \quad  \text{ for } m \geq 1,
\end{equation}
\noindent
(see Theorem 3.3 in \cite{dinardo-2010a}), therefore in the current situation
\begin{equation}
\mathbb{E} \left[ X_{1}(X_{1}+iY_{1})^{j}  \right] = 
\begin{cases} 
0 & \quad \text{ if } j \text{ is even} \\
\kappa(2m) & \quad \text{ if } j = 2m + 1 \text{ is odd}.
\end{cases}
\end{equation}
\noindent
On the other hand 
\begin{eqnarray*}
\mathbb{E} \left[ X_{1}(X_{1}+z)^{n} \right]  & = & 
\mathbb{E} \left[ (X_{1}+z)^{n+1} - z (X_{1}+z)^{n} \right]  \\
& = & P_{n+1}(z) - z P_{n}(z).
\end{eqnarray*}
\noindent
Replacing in \eqref{nice-010} yields the result.
\end{proof}

\medskip

Recall that a random variable has a Laplace distribution if its
distribution function is 
\begin{equation}
f_{L}(x) = \tfrac{1}{2}e^{-|x|}.
\end{equation}

\smallskip

Assume $X_{\mu}$ has a distribution $f_{\mu}$ defined in 
\eqref{semicircular0} and moment generating function given by \eqref{factor}.
The next lemma constructs a random variable $Y_{\mu}$ conjugate to $X_{\mu}$.

\begin{lemma}
Let $Y_{\mu,n}$ be a random variable defined by 
\begin{equation}
Y_{\mu,n} = \sum_{k=1}^{n} \frac{L_{k}}{j_{\mu,k}}
\end{equation}
\noindent
where $\{ L_{k}: \, k \in \mathbb{N} \}$ is a sequence of independent, 
identically distributed  Laplace 
random variables. Then 
$\lim\limits_{n \to \infty} Y_{\mu,n} = Y_{\mu}$ exists and is a 
random variable with continuous probability density. Moreover, the moment 
generating function of $iY_{\mu}$ is 
\begin{equation}
\mathbb{E} \left[ e^{it Y_{\mu}} \right]
= \prod_{k=1}^{\infty} \left( 1 + \frac{t^{2}}
{j_{\mu,k}^{2}} \right)^{-1}.
\label{recigf0}
\end{equation}
\noindent
the reciprocal of the moment generating function of $f_{\mu}$ given in 
\eqref{factor}. 
\end{lemma}
\begin{proof}
The characteristic function of a Laplace random variable $iL_{k}/j_{\mu,k}$ 
is 
\begin{equation}
\varphi_{iL_{k}}(t) = \frac{1}{1+ \frac{t^{2}}{j_{\mu,k}^{2}}}.
\end{equation}
\noindent
The values
\begin{equation}
\mathbb{E} \left[ \frac{L_{k}}{j_{\mu,k}} \right] = 0, \text{ and }
\mathbb{E} \left[ \frac{L_{k}^{2}}{j_{\mu,k}^{2}} \right] = 
\frac{2}{j_{\mu,k}^{2}}, 
\end{equation}
\noindent
guarantee the convergence of the series
\begin{equation}
\sum_{k=1}^{\infty} \mathbb{E} \left[ \frac{L_{k}}{j_{\mu,k}} \right] 
\text{ and }
\sum_{k=1}^{\infty} \mathbb{E} \left[ \frac{L_{k}^{2}}{j_{\mu,k}^{2}} 
\right].
\end{equation}
\noindent
(The last series evaluates to $1/(2 \mu + 2)$). This ensures the existence 
of the limit defining $Y_{\mu}$ (see \cite{jessen-1935a} for details). The 
continuity of the limiting probability density $Y_{\mu}$ is ensured by the 
fact that at least one term (in fact all) in the defining sum has a 
continuous probability density that is of bounded variation.
\end{proof}

\begin{note}
In the case $X_{\mu} \sim f_{\mu}$ is independent of $Y_{\mu}$, then the 
conjugacy property states that if $h$ is an analytic 
function in a neighborhood $\mathcal{O}$ of the origin, then 
\begin{equation}
\mathbb{E} \left[h(z + X_{\mu} + i Y_{\mu}) \right] 
= h(z), \quad \text{ for } z \in 
\mathcal{O}.
\end{equation}
\noindent
In particular
\begin{equation}
\label{cancellation0}
\mathbb{E}\left[(X_{\mu}+i Y_{\mu})^{n} \right] = 
\begin{cases}
1 & \quad \text{ if } n = 0, \\
0 & \quad \text{ otherwise}.
\end{cases}
\end{equation}
\end{note}

\begin{note}
In the special case $\mu = n/2-1$ for $n \in \mathbb{N}, \, n \geq 3$, the 
function \eqref{recigf0} has been characterized in \cite{ciesielksi-1962a} 
as the moment generating function of the total time $T_{n}$ spent in the 
sphere $S^{n-1}$ by an $n$-dimensional Brownian motion starting at the 
origin.
\end{note}

\section{The Narayana polynomials and the sequence $A_{n}$} 
\label{S:narayana} 

The result of Theorem \ref{thm-recurrence0} is now applied to a random 
variable $X \sim f_{1}$. In this case the polynomials $P_{n}$ 
correspondi, up to a change of variable, to 
the Narayama polynomials $\mathcal{N}_{n}$.  The 
recurrence established by M. Lasalle 
comes from the results in Section \ref{S:prob}. In particular, this 
provides an interpretation of the sequence $\{ A_{n} \}$ in terms of 
cumulants and the Bessel zeta function. 

Recall the distribution function $f_{1}$
\begin{equation}
\label{form-f1}
f_{1}(x) = 
\begin{cases}
2\sqrt{1-x^{2}}/\pi, & \quad \text{ for } |x| \leq 1  \\
0, & \quad \text{ otherwise}.
\end{cases}
\end{equation}

\begin{lemma}
\label{nara-exp0}
Let $X \sim f_{1}$. The Narayana polynomials appear as the moments
\begin{equation}
\mathcal{N}_{r}(z) = 
\mathbb{E} \left[ \left( 1 + z + 2 \sqrt{z} X \right)^{r-1} \right],
\end{equation}
\noindent
for $r \geq 1$. 
\end{lemma}
\begin{proof}
The binomial theorem gives 
\begin{equation*}
\mathbb{E} \left[(1 + z + 2 \sqrt{z} X)^{r-1} \right]  =  
\sum_{j=0}^{r-1} \binom{r-1}{j} (z+1)^{r-1-j} z^{j/2} 
\mathbb{E} \left[(2X)^{j} \right].
\end{equation*}
\noindent
The result now follows from \eqref{exp-catalan} and \eqref{nara-cata0}.
\end{proof}

In order to apply Theorem \ref{thm-recurrence0} consider the identities
\begin{eqnarray}
\mathcal{N}_{r}(z) & = & 
\mathbb{E}\left[ \left( 1 + z + 2 \sqrt{z} X \right)^{r-1} \right] \\
& = & ( 2 \sqrt{z} )^{r-1} 
\mathbb{E}\left[ \left(X + z_{*} \right)^{r-1} \right] \nonumber \\
& = & (2 \sqrt{z})^{r-1} P_{r-1}(z_{*}),
\nonumber
\end{eqnarray}
\noindent 
with 
\begin{equation}
z_{*} = \frac{1+z}{ 2 \sqrt{z}}. 
\end{equation}

The recurrence \eqref{recu-00} applied to the polynomial $P_{n}(z_{*})$ 
yields
\begin{equation}
\frac{\mathcal{N}_{n+2}(z)}{(2 \sqrt{z})^{n+1}} - 
\frac{(1+z)}{2 \sqrt{z}} \frac{\mathcal{N}_{n+1}(z)}{( 2 \sqrt{z})^{n}} = 
\sum_{m \geq 1} \binom{n}{2m-1} \kappa(2m) 
\frac{\mathcal{N}_{n-2m+2}(z)}{( 2 \sqrt{z})^{n-2m+1}}
\end{equation}
\noindent
that reduces to 
\begin{equation}
\label{rec-nara1}
(1+z)\mathcal{N}_{r}(z) - \mathcal{N}_{r+1}(z) = 
- \sum_{m \geq 1} \binom{r-1}{2m-1} \kappa(2m) 2^{2m} z^{m} 
\mathcal{N}_{r+1-2m}(z),
\end{equation}
\noindent
by using $r = n+1$. This recurrence has the form of \eqref{recu-lasalle}. 

\begin{theorem}
\label{thm-mu1}
Let $X \sim f_{1}$. Then the coefficients $A_{n}$ in Definition \ref{def0-A}
are given by 
\begin{equation}
A_{n} = (-1)^{n+1} \kappa(2n) 2^{2n}.
\end{equation}
\end{theorem}

The expression in \eqref{cumu-zeta} gives the next result.

\begin{corollary}
Let 
\begin{equation}
\zeta_{\mu}(s) = \sum_{k=1}^{\infty} \frac{1}{j_{\mu,k}^{s}}
\end{equation}
\noindent
be the Bessel zeta function. Then the 
coefficients $A_{n}$ are given by
\begin{equation*}
A_{n}  =  2^{2n+1} (2n-1)! \, \zeta_{1}(2n).
\label{An-zeta1}
\end{equation*}
\end{corollary}

The scaled coefficients $a_{n}$ are now expressed in terms of the Bessel 
zeta function.

\begin{corollary}
The coefficients $a_{n}$ are given by 
\begin{equation}
a_{n} = 2^{2n+1} (n+1)! (n-1)! \, \zeta_{1}(2n).
\label{an-zeta}
\end{equation}
\end{corollary}

\begin{note}
This expression for the coefficients and the recurrence 
\begin{equation}
(n+ \mu) \zeta_{\mu}(2n) = \sum_{r=1}^{n-1} \zeta_{\mu}(2r) 
\zeta_{\mu}(2n-2r).
\label{recu-zeta1}
\end{equation}
given in \cite{elizalde-1993a}, provides a new proof of the 
recurrence in Proposition \eqref{recu-fora}.
\end{note}

\section{The generalized Narayana polynomials} 
\label{S:generalized} 

The Narayama polynomials $\mathcal{N}_{r}(z)$, defined in 
\eqref{nara-poly-def}, have been expressed as the moments
\begin{equation}
\mathcal{N}_{r}(z) 
= \mathbb{E} \left[ \left( 1 + z + 2 \sqrt{z} X \right)^{r-1} \right],
\label{nara-exp1}
\end{equation}
\noindent
for $r \geq 1$. Here $X$ is a random variable with distribution function 
$f_{1}$. This suggests the extension
\begin{equation}
\mathcal{N}_{n}^{\mu}(z) = 
\mathbb{E} \left[ \left( 1 + z + 2 \sqrt{z} X \right)^{n-1} \right],
\label{nara-exp2}
\end{equation}
\noindent
with $X \sim f_{\mu}$. Therefore, $\mathcal{N}_{n} = \mathcal{N}_{n}^{1}$. 

\begin{note}
The same argument given in \eqref{rec-nara1} gives the recurrence
\begin{equation}
\label{rec-nara2}
(1+z)\mathcal{N}_{r}^{\mu}(z) - \mathcal{N}_{r+1}^{\mu}(z) = 
- \sum_{m \geq 1} \binom{r-1}{2m-1} \kappa(2m) 2^{2m} z^{m} 
\mathcal{N}_{r+1-2m}^{\mu}(z),
\end{equation}
\noindent
where $\kappa(2n)$ are the cumulants of $X \sim f_{\mu}$. Theorem \ref{thm-mu1}
gives an expression for the generalization of the Lasalle numbers:
\begin{equation}
A_{n}^{\mu} := (-1)^{n+1} \kappa(2n) 2^{2n}
\end{equation}
\noindent
and the corresponding expression in terms of the Bessel zeta function:
\begin{equation}
A_{n}^{\mu} := 2^{2n+1} (2n-1)! \zeta_{\mu}(2n).
\end{equation}
\end{note}

The generalized Narayana polynomials are now expressed in terms of the Gegenbauer polynomials
$C_{n}^{\mu}(x)$ defined by the generating function 
\begin{equation}
\sum_{n=0}^{\infty} C_{n}^{\mu}(x)t^{n} = (1 - 2xt + t^{2})^{-\mu}.
\end{equation}
\noindent
These polynomial admit several hypergeometric representations:
\begin{eqnarray}
\quad C_{n}^{\mu}(x) & = & 
\frac{(2 \mu)_{n}}{n!} {_{2}F_{1}} \left( -n, n + 2 \mu; \mu + \tfrac{1}{2}; 
\frac{1-x}{2} \right) \label{gegen-hyper} \\
& = & \frac{2^{n}(\mu)_{n}}{n!} (x-1)^{n} \, 
{_{2}F_{1}} \left( -n, -n - \mu + \tfrac{1}{2}; -2n -2 \mu + 1; 
\frac{2}{1-x} \right) \nonumber  \\
& = & \frac{(2 \mu)_{n}}{n!} \left( \frac{x+1}{2} \right)^{n} \, 
{_{2}F_{1}} \left( -n, -n - \mu + \tfrac{1}{2}; \mu + \tfrac{1}{2}; 
\frac{x-1}{x+1} \right). \nonumber
\end{eqnarray}

The connection between Narayana and Gegenbauer polynomials comes from the 
expression for $C_{n}^{\mu}(z)$ given in the next proposition.

\begin{prop}
The Gegenbauer polynomials are given by 
\begin{equation}
C_{n}^{\mu}(z) = \frac{(2 \mu)_{n}}{n!} 
\mathbb{E} \left[ \left( z + \sqrt{z^{2}-1} X_{\mu - 1/2} \right)^{n} \right].
\end{equation}
\end{prop}
\begin{proof}
The Laplace integral representation 
\begin{equation}
C_{n}^{\mu}(\cos \theta) = \frac{\Gamma(n + 2 \mu)}{2^{2 \mu - 1} n! 
\Gamma^{2}(\mu)} 
\int_{0}^{\pi}  \left( \cos \theta + i \sin \theta \cos \phi \right)^{n} 
\sin^{2 \mu - 1} \phi \, d \phi 
\end{equation}
\noindent
appears as Theorem $6.7.4$ in \cite{andrews-1999a}. The change of variables 
$z = \cos \theta$ and $X = \cos \phi$ gives 
\begin{eqnarray*}
C_{n}^{\mu}(z) & = & \frac{\Gamma(n + 2 \mu)}{2^{2 \mu} \, n! \Gamma^{2}(\mu)}
\int_{-1}^{1} \left( z + \sqrt{z^{2}-1}X \right)^{n} \, (1 - X^{2})^{\mu-1} 
\, dX \\
& = & \frac{(2 \mu)_{n}}{n!} \mathbb{E} \left[ \left( z + \sqrt{z^{2}-1} 
X_{\mu - 1/2} \right)^{n} \right],
\end{eqnarray*}
\noindent
as claimed. Since this is a polynomial identity in $z$, it can be extended
to all $z \in \mathbb{C}$.
\end{proof}

\begin{theorem}
The Gegenbauer polynomial $C_{n}^{\mu}$ and the generalized polynomial 
$\mathcal{N}_{n}^{\mu}$ satisfy the relation
\begin{equation}
\mathcal{N}_{n+1}^{\mu}(z) = \frac{n!}{(2 \mu+1)_{n}} (1-z)^{n} 
C_{n}^{\mu + \tfrac{1}{2}} \left( \frac{1 + z}{1-z} \right).
\end{equation}
\end{theorem}
\begin{proof}
Introduce the variable 
\begin{equation}
Z = \frac{1+z}{1-z}
\end{equation}
\noindent 
so that 
\begin{equation}
z = \frac{Z-1}{Z+1} \text{ and }
\frac{Z}{\sqrt{Z^{2}-1}} = \frac{1+z}{2 \sqrt{z}}.
\end{equation}
Then 
\begin{eqnarray*}
C_{n}^{\mu + \tfrac{1}{2}} \left( \frac{1+z}{1-z} \right) & = & 
\frac{(2 \mu + 1)_{n}}{n!} \left( \frac{2 \sqrt{z}}{1-z} \right)^{n} 
\mathbb{E} \left[ \left( \frac{1+z}{2 \sqrt{z}} + X_{\mu} \right)^{n} \right] \\
& = & \frac{(2 \mu + 1)_{n}}{n! \, (1-z)^{n}} 
\mathbb{E} \left[ \left( 1+z + 2 \sqrt{z} X_{\mu} \right)^{n} \right] \\
& = & \frac{(2 \mu+1)_{n}}{n! \, (1-z)^{n}}  \mathcal{N}_{n+1}^{\mu}(z),
\end{eqnarray*}
\noindent
using $Z^{2}-1 = 4z/(1-z)^{2}$.
\end{proof}

The expression \eqref{gegen-hyper} now provides hypergeometric expressions 
for the original Narayana polynomials
\begin{equation}
\mathcal{N}_{n+1}(z) = \frac{2(1-z)^{n}}{(n+2)(n+1)} C_{n}^{3/2} 
\left( \frac{1+z}{1-z} \right).
\end{equation}

\begin{corollary}
The Narayana polynomials are given by 
\begin{eqnarray}
\mathcal{N}_{n+1}(z) & = &  (1-z)^{n} {_{2}F_{1}} 
\left(-n, n+3; 2; \frac{z}{z-1} \right) \\
& = & \frac{(2n+2)!}{(n+2)! \, (n+1)!} z^{n} 
{_{2}F_{1}} \left( -n, -n-1; -2n-2; \frac{z-1}{z} \right) \nonumber \\
& = & {_{2}F_{1}} ( -n, -n-1;2;z). \nonumber
\end{eqnarray}
\noindent
This yields the representation as finite sums
\begin{eqnarray}
\mathcal{N}_{n+1}(z) & = & 
\sum_{k=0}^{n} \frac{1}{k+1} \binom{n}{k} \binom{n+k+2}{k} z^{k} 
(1-z)^{n-k} \\
& = & \frac{1}{n+1}
\sum_{k=0}^{n}  \binom{n+1}{k} \binom{2n+2-k}{n-k} z^{n-k} 
(1-z)^{k} \nonumber \\
& = & \frac{1}{n+1}
\sum_{k=0}^{n}  \binom{n+1}{k+1} \binom{n+1}{k} z^{k}. \nonumber
\end{eqnarray}
\end{corollary}

Note that the first two expressions coincide up to the change of 
summation variable $k \to n-k$ while the third identity is nothing but
\eqref{nara-poly-def}. 

\begin{note}
The representation 
\begin{equation}
C_{n}^{\mu}(z) = \frac{(\mu)_{n}}{n!} (2x)^{n} 
{_{2}F_{1}} \left( - \frac{n}{2}, \frac{1-n}{2}; 1 - n - \mu; \frac{1}{x^{2}}
\right)
\end{equation}
\noindent
that appears in as $6.4.12$ in \cite{andrews-1999a}, gives the expression 
\begin{equation}
\mathcal{N}_{n+1}(z) = \frac{(2n+2)!}{(n+1)! \, (n+2)!} 
\left( \frac{1+z}{2} \right)^{n} 
{_{2}F_{1}} \left( - \frac{n}{2}, \frac{1-n}{2}; -n - \frac{1}{2}; 
\left( \frac{1-z}{1+z} \right)^{2} \right)
\end{equation}
\noindent
equal to the finite sum representation
\begin{equation}
\mathcal{N}_{n+1}(z) = \frac{1}{2^{n-1} (n+2)} 
\sum_{k=0}^{\lfloor n/2 \rfloor} (-1)^{k} \binom{n}{k} 
\binom{2n+1-2k}{n-2k} (1-z)^{2k} (1+z)^{n-2k}.
\end{equation}
\end{note}

\begin{note}
The polynomials 
$S_{n}(z) = z \mathcal{N}_{n}^{1}(z)$  
satisfy the symmetry identity 
\begin{equation}
S_{n}(z) = z^{n+1} S_{n}(z^{-1}).
\end{equation}
\noindent
These polynomials were expressed in \cite{mansour-2009a} as 
\begin{equation}
S_{n}(z) = (z-1)^{n+1} \int_{0}^{z/(z-1)} P_{n}(2x-1) \, dx 
\label{nara-int0}
\end{equation}
\noindent
where $P_{n}(x) = C_{n}^{1/2}(x)$ are the Legendre polynomials. An equivalent 
formulation is provided next.
\end{note}

\begin{theorem}
The polynomials $S_{n}(z)$ are given by
\begin{equation*}
S_{n}(z) = \frac{1}{2^{n+1}} \sum_{k=0}^{\lfloor{n/2 \rfloor}} 
\frac{(-1)^{k}}{n+1-k} \binom{2n-2k}{n-k} \binom{n+1-k}{k} 
(z-1)^{2k} (z+1)^{n+1-2k}.
\end{equation*}
\end{theorem}
\begin{proof}
The integration rule 
\begin{equation}
\int C_{n}^{\mu}(x) \, dx = \frac{1}{2(\mu-1)} C_{n+1}^{\mu-1}(x)
\end{equation}
\noindent
implies
\begin{equation}
\int_{0}^{z/(z-1)} C_{n}^{1/2}(2x-1) \, dx = - \frac{1}{2} C_{n+1}^{-1/2}
\left( \frac{z+1}{z-1} \right),
\end{equation}
\noindent
since the generating function 
\begin{equation}
\sum_{n=0}^{\infty} t^{n} C_{n}^{-1/2}(z) = (1 - 2zt+t^{2})^{1/2}
\end{equation}
\noindent
gives $C_{n+1}^{-1/2}(-1) = 0$ for $n > 1$. Then \eqref{nara-int0} yields 
\begin{equation}
\label{nara-int1}
S_{n}(z) = - \frac{1}{2} (z-1)^{n+1} C_{n+1}^{-1/2} 
\left( \frac{z+1}{z-1} \right).
\end{equation}

A classical formula for the 
Gegenbauer polynomials states
\begin{equation}
C_{n}^{\mu}(z) = \sum_{k=0}^{\lfloor{ n/2 \rfloor}} \frac{(-1)^{k}}{k!}
\frac{(\mu)_{n-k}}{(n-2k)!} (2z)^{n-2k}
\end{equation}
\noindent
and the identity 
\begin{equation*}
\left( - \frac{1}{2} \right)_{k} = - \frac{1}{2^{2k-1}} \frac{(2k-2)!}{(k-1)!}
\end{equation*}
\noindent
produce 
\begin{equation}
C_{n}^{-1/2}(z) = \frac{1}{2^{n-1}}
\sum_{k=0}^{\lfloor{ n/2 \rfloor}} \frac{(-1)^{k+1}}{n-k}
\binom{2n-2k-2}{n-k-1} \binom{n-k}{k} z^{n-2k}.
\end{equation}

The result now follows from \eqref{nara-int1}.
\end{proof}

\section{The generalization of the numbers $a_{n}$} 
\label{S:special} 

The terms forming the original suggestion of Zeilberger 
\begin{equation}
a_{n} = \frac{2A_{n}}{C_{n}}
\end{equation}
\noindent
have been given a probabilistic interpretation: let $X$ be a random variable 
with a symmetric beta distribution function with parameter $\mu = 1$ given 
explicitly in \eqref{form-f1}. The numerator $A_{n}$ is 
\begin{equation}
A_{n} = (-1)^{n+1} \kappa(2n) 2^{2n}
\end{equation}
\noindent
where $\kappa(2n)$ is the even-order cumulant of the scaled random 
variable $X_{*} = 2X$. The denominator $C_{n}$ is interpreted as the even-order 
moment of $X_{*}$: 
\begin{equation}
C_{n} = \mathbb{E} \left[ X_{*}^{2n} \right].
\end{equation}
\noindent
These notions are used now to define an extension of the coefficients $a_{n}$.

\begin{definition}
Let $X$ be a random variable with vanishing odd cumulants. The 
numbers $a_{n}(\mu)$ are defined by 
\begin{equation}
a_{n}(\mu) = \frac{2 (-1)^{n+1} \kappa(2n)}{\mathbb{E}\left[ X_{*}^{2n} 
\right]}
\end{equation}
\end{definition}

\noindent
In the special case $X_{*} = 2X$ with $X \sim f_{\mu}$, these numbers 
are computed using the cumulants
\begin{equation}
\kappa_{\mu}(2n) = (-1)^{n+1} 2^{2n+1} (2n-1)! \zeta_{\mu}(2n)
\end{equation}
\noindent
and the even order moments
\begin{equation}
\mathbb{E}\left[ X_{*}^{2n} \right] = \frac{(2n)!}{n!} \frac{1}{(\mu+1)_{n}}
\end{equation}
\noindent 
to produce
\begin{equation}
a_{n}(\mu) =  
2^{2n+1} \, (n-1)! \, (\mu + 1)_{n} \, \zeta_{\mu}(2n).
\label{anmu-form1}
\end{equation}
The value 
\begin{equation}
\zeta_{\mu}(2) = \frac{1}{4(\mu+1)}
\end{equation}
\noindent
yields the initial condition $a_{1}(\mu) = 2$. 

\smallskip

The recurrence \eqref{recu-zeta1} now provides the next result. Recall that 
when $x$ is not necessarily a  positive integer, the binomial coefficient is 
given by 
\begin{equation}
\binom{x}{k} = \frac{\Gamma(x+1)}{\Gamma(x-k+1) \, k!}.
\end{equation}

\smallskip

\begin{prop}
The coefficients $a_{n}(\mu)$ satisfy the recurrence 
\begin{equation}
a_{n}(\mu) = \frac{1}{2 \binom{n+\mu-1}{n-1}}
\sum_{k=1}^{n-1} \binom{n+\mu-1}{n-k-1} \binom{n+\mu-1}{k-1}
a_{k}(\mu) a_{n- k}(\mu),
\label{gen-recu}
\end{equation}
\noindent
with initial condition $a_{1}(\mu)  = 2$.
\end{prop}
\begin{proof}
Start with the convolution identity for Bessel zeta functions \eqref{recu-zeta1}
and replace each zeta function by its expression in terms of $a_{n}(\mu)$ 
from \eqref{anmu-form1}, which gives
\begin{multline*}
(n+\mu) \frac{a_{n}(\mu)}{2^{2n+1} (n-1)! (\mu+1)_{n}} = \\
\sum_{k=1}^{n-1} \frac{a_{k}(\mu)}{2^{2k+1} (k-1)! (\mu+1)_{k}} \, 
\frac{a_{n-k}(\mu)}{2^{2n-2k+1} (n-k-1)! (\mu+1)_{n-k}} 
\end{multline*}
\noindent
and after simplification 
\begin{equation*}
a_{n}(\mu) = \frac{1}{2(n+\mu)} 
\sum_{k=1}^{n-1} \frac{(n-1)!}{(k-1)! (n-k-1)!} 
\frac{(\mu+1)_{n}}{(\mu+1)_{k} (\mu+1)_{n-k}} a_{k}(\mu) a_{n-k}(\mu).
\end{equation*}
\noindent
The resut now follows by elementary algebra. 
\end{proof}

\begin{note}
In the case $\mu = 1$, the recurrence \eqref{gen-recu} becomes \eqref{rec-11}
and the coefficients $a_{n}(1)$ are the original numbers $a_{n}$.
\end{note}

\begin{note}
The recurrence \eqref{gen-recu} can be written as 
\begin{equation*}
a_{n}(\mu) = \frac{1}{2} 
\sum_{k=1}^{n-1} 
\frac{\Gamma(n) \Gamma(\mu+1) \Gamma(n+ \mu)}
{\Gamma(\mu+k+1) \Gamma(n+ \mu-k+1) 
\Gamma(n-k) \Gamma(k)} a_{k}(\mu) a_{n-k}(\mu).
\end{equation*}
\end{note}

\begin{theorem}
The coefficients $a_{n}(\mu)$ are 
positive and increasing for
$n \geq \lfloor \frac{\mu+3}{2} \rfloor$. 
\end{theorem}
\begin{proof}
The positivity is clear from \eqref{anmu-form1}. Now take the terms 
corresponding to $k=1$ and $k=n-1$ in \eqref{gen-recu} to obtain 
\begin{equation}
a_{n}(\mu) \geq \frac{n-1}{\mu+1} a_{1}(\mu) a_{n-1}(\mu)  =
\frac{2(n-1)}{\mu+1} a_{n-1}(\mu). 
\end{equation}
\noindent
This yields
\begin{equation}
a_{n}(\mu) - a_{n-1}(\mu) \geq \frac{2n-3 - \mu}{\mu+1} a_{n-1}(\mu)
\end{equation}
\noindent
and the result follows.
\end{proof}

Some other special cases are considered next. 

\medskip

\noindent
\textbf{The case $\mu = 0$}. In this situation the distribution is the 
arcsine distribution given by
\begin{equation}
\label{form-f0}
f_{0}(x) = 
\begin{cases}
\frac{1}{\pi} \, \, \frac{1}{\sqrt{1-x^{2}}}, & \quad \text{ for } |x| \leq 1  \\
0, & \quad \text{ otherwise}.
\end{cases}
\end{equation}
\noindent
By the recurrence on the $\zeta_{0}$ function, the 
coefficients 
\begin{equation}
a_{n}(0) = 2^{2n} (n-1)! \, n! \, \zeta_{0}(2n)
\end{equation}
\noindent
satisfy the recurrence 
\begin{equation}
a_{n}(0) = \frac{1}{2} \sum_{k=1}^{n-1} \binom{n-1}{k} \binom{n-1}{k-1} 
a_{k}(0) a_{n-k}(0)
\label{recu-a88}
\end{equation}
\noindent
with $a_{1}(0) =2$. Now define as Lasalle 
$b_{n} = \tfrac{1}{2}a_{n}(0)$ and then 
\eqref{recu-a88} becomes 
\begin{eqnarray}
b_{n} & = &  \sum_{k=1}^{n-1} \binom{n-1}{k} \binom{n-1}{k-1} b_{k}b_{n-k}, 
\label{recu-b} \\
b_{1} &  = & 1. \nonumber 
\end{eqnarray}
\noindent
In particular $b_{n}$ is a positive integer. 

\medskip

The following comments are obtained by an analysis similar to that for $a_{n}$.

\begin{note}
The recurrence 
\begin{equation*}
\sum_{j=1}^{n} (-1)^{j-1} \binom{n}{j} \binom{n-1}{j-1}b_{j} = 1 
\end{equation*}
\noindent
gives the generating function
\begin{equation*}
\sum_{j=1}^{\infty} \frac{(-1)^{j-1} b_{j}}{j!} \frac{x^{2j-2}}{(j-1)!} 
= \frac{I_{1}(2x)}{x \, I_{0}(2x)} = 
\frac{1}{2x} \frac{d}{dx} \log I_{0}(2x).
\end{equation*}
\end{note}

\begin{note}
The sequence $b_{n}$ admits a determinant representation 
$b_{n} = \text{det}(M_{n})$, where 
\begin{equation}
M_{n} = 
\begin{pmatrix} 
1 & \binom{1}{1} \binom{1-1}{1-1} & 0 & 0 & \cdots & 0 \\
1 & \binom{2}{1} \binom{2-1}{1-1} & \binom{2}{2} \binom{2-1}{2-1}
 & 0 & \cdots & 0 \\
1 & \binom{3}{1} \binom{3-1}{1-1} & \binom{3}{2} \binom{3-1}{2-1}
 & \binom{3}{3} \binom{3-1}{3-1} & \cdots & 0 \\
\cdots  & \cdots  & \cdots 
 & \cdots  & \cdots & \cdots  \\
\cdots  & \cdots  & \cdots 
 & \cdots  & \cdots & \cdots  \\
1 & \binom{n}{1} \binom{n-1}{1-1} & \binom{n}{2} \binom{n-1}{2-1}
 & \binom{n}{3} \binom{n-1}{3-1} & \cdots & \binom{n}{n-1} \binom{n-1}{n-2} 
\end{pmatrix}
\end{equation}
\end{note}

\begin{note}
The identity $I_{2}(x)  = I_{0}(x) - \frac{2}{x} I_{1}(x)$ is expressed as 
\begin{equation}
\frac{I_{1}(2x)}{x I_{0}(2x)} \left[ 1 + \frac{1}{2} x^{2} \frac{2 I_{2}(2x)}
{x I_{1}(2x)} \right] = 1
\end{equation}
\noindent
provides the relation
\begin{equation}
b_{n} = \frac{1}{2} \sum_{j=1}^{n-1} \binom{n-1}{j} \binom{n}{j-1} b_{j} 
a_{n-j}.
\end{equation}
\end{note}

\medskip

\noindent
\textbf{The case $\mu = \tfrac{1}{2}$}. In this situation the distribution is 
the uniform distribution on $[-1,1]$ with even moments 
\begin{equation}
\mathbb{E} X_{*}^{2n} = \frac{2^{2n}}{2n+1}
\end{equation}
\noindent
and vanishing odd moments. The sequence of cumulants is 
\begin{equation}
\kappa_{1/2}(2n) = 2 (-1)^{n+1} (2n-1)! \, \zeta_{1/2}(2n)
\end{equation}
\noindent
where the Bessel zeta function is 
\begin{equation}
\zeta_{1/2}(2n) = \sum_{k=1}^{\infty} \frac{1}{\pi^{2n} \, k^{2n}} = 
\frac{1}{\pi^{2n}} \zeta(2n) = 
\frac{2^{2n-1}}{(2n)!} |B_{2n}|,
\end{equation}
\noindent
where $B_{n}$ are the Bernoulli numbers. This follows from the identity 
\begin{equation}
J_{1/2}(x) = \sqrt{\frac{2}{\pi x}} \sin x.
\end{equation}
\noindent
This yields
\begin{equation}
\kappa_{1/2}(2n) = 2^{2n} \frac{B_{2n}}{2n} \text{ and } \kappa_{1/2}(2n+1) 
= 0, 
\end{equation}
\noindent 
with $\kappa_{1/2}(0) = 0$. These are the coefficients of 
$u^{n}/n!$ in the cumulant moment generating function 
\begin{equation}
\log \varphi_{1/2}(u) = \log \frac{\sinh u}{u} = 
\frac{1}{6}u^{2} - \frac{1}{180}u^{4}  + \frac{1}{2835}u^{6} + \cdots.
\end{equation}
\noindent
Finally, the corresponding sequence 
\begin{equation}
a_{n}\left( \tfrac{1}{2} \right) =  
\frac{2 (-1)^{n+1} \kappa(2n)}{\mathbb{E} \left[ X_{*}^{2n} \right]}
\end{equation}
\noindent
is given by 
\begin{equation}
a_{n} \left( \tfrac{1}{2} \right)  =   2^{2n} \frac{2n+1}{n} |B_{2n}|.
\end{equation}
\noindent
The first few terms are 
\begin{equation}
a_{1} \left( \tfrac{1}{2} \right)  = 2, \, 
a_{2} \left( \tfrac{1}{2} \right)  = \frac{4}{3}, \, 
a_{3} \left( \tfrac{1}{2} \right)  = \frac{32}{9}, \, 
a_{4} \left( \tfrac{1}{2} \right)  = \frac{96}{5}, \, 
a_{5} \left( \tfrac{1}{2} \right)  = \frac{512}{3}, 
\end{equation}
\noindent
as expected, this is an increasing sequence for $n \geq 3$. The 
convolution identity 
\eqref{recu-zeta1} for Bessel zeta functions 
gives the well-known quadratic relation for the 
Bernoulli numbers
\begin{equation}
\sum_{k=1}^{n-1} \binom{2n}{2k} B_{2k} B_{2n-2k} = - (2n+1)B_{2n}, 
\quad \text{ for } n > 1.
\end{equation}
\noindent
Moreover, the moment-cumulants relation \eqref{convolution} gives, replacing 
$n$ by $2n$ and after simplification, the other well-known identity 
\begin{equation}
\sum_{j=1}^{n} \binom{2n+1}{2j} 2^{2j}B_{2j} = 2n, \text{ for } n \geq 1.
\end{equation}

\begin{note}
The generating function of the sequence $a_{n}\left( \tfrac{1}{2} \right)$ 
is given by 
\begin{equation*}
\frac{I_{3/2}(x)}{x I_{1/2}(x)} =  \frac{x \tanh x  - 1}{x^{2}} = 
\sum_{j=1}^{\infty} \frac{(-1)^{j-1} 
2a_{j}\left( \tfrac{1}{2} \right)}{(2j+1)(2j-1)!}x^{2j-2}.
\end{equation*}
\end{note}

\medskip

\noindent
\textbf{The limiting case $\mu = -\tfrac{1}{2}$} has the probability 
distribution 
\begin{equation}
f_{-1/2}(x) = \frac{1}{2} \delta(x-1) + \frac{1}{2} \delta(x+1)
\end{equation}
(the discrete Rademacher distribution). For a Rademacher random variable 
$X$, the odd moments of $X_{*} = 2X$ vanish while the even order moments are 
\begin{equation}
\mathbb{E} \left[ X_{*}^{2n} \right] = 2^{2n}.
\end{equation}
\noindent
Therefore 
\begin{equation}
\kappa_{-1/2}(2n) = (-1)^{n+1} 2^{2n+1} (2n-1)! \, \zeta_{-1/2}(2n).
\end{equation}
\noindent
The identity 
\begin{equation}
J_{-1/2}(x) = \sqrt{ \frac{2}{\pi x}} \cos x
\end{equation}
\noindent
shows that $j_{k,-1/2} = (2k-1)\pi/2$ and therefore 
\begin{equation}
\zeta_{-1/2}(2n) = \sum_{k=1}^{\infty} \frac{2^{2n}}{\pi^{2n} (2k-1)^{2n}} 
= \frac{2^{2n}-1}{\pi^{2n}} \zeta(2n).
\end{equation}
\noindent
The expression for $\kappa_{-1/2}(2n)$ may be simplified by the relation 
\begin{equation}
E_{n} = - \frac{2}{n+1}(2^{n+1}-1)B_{n+1}
\end{equation}
\noindent
between the Euler numbers $E_{n}$ and the Bernoulli numbers. It follows that 
\begin{equation}
\kappa_{-1/2}(2n) =  -2^{4n-1}E_{2n-1}.
\end{equation}
\noindent
The corresponding sequence $a_{n} \left( - \tfrac{1}{2} \right)$ is 
now given by 
\begin{equation}
a_{n} \left( - \tfrac{1}{2} \right) = (-1)^{n} 2^{2n} E_{2n-1}
\end{equation}
\noindent
and its first few values are 
\begin{equation*}
a_{1} \left( -\tfrac{1}{2} \right)  = 2, \, 
a_{2} \left( -\tfrac{1}{2} \right)  = 4, \, 
a_{3} \left( -\tfrac{1}{2} \right)  = 32, \, 
a_{4} \left( -\tfrac{1}{2} \right)  = 544,  \, 
a_{5} \left( -\tfrac{1}{2} \right)  = 15872,  \, 
\end{equation*}

\begin{note}
The generating function of the sequence 
$a_{n} \left( - \tfrac{1}{2} \right)$ is given by 
\begin{equation*}
\frac{I_{1/2}(x)}{x I_{-1/2}(x)} =  \frac{\tanh x}{x} = 
\sum_{j=1}^{\infty} \frac{(-1)^{j-1} 
2a_{j}\left( - \tfrac{1}{2} \right)}{(2j-1)!}x^{2j-2}.
\end{equation*}
\end{note}

\begin{note}
The convolution identity \eqref{recu-zeta1} yields the well-known quadratic 
recurrence relation
\begin{equation}
\sum_{k=1}^{n-1} \binom{2n-2}{2k-1} E_{2k-1} E_{2n-2k-1} = 
2 E_{2n-1}, \text{ for } n > 1,
\end{equation}
\noindent
and the moment-cumulant relation \eqref{convolution} gives the other 
well-known identity
\begin{equation}
\sum_{k=1}^{n} \binom{2n-1}{2k-1} 2^{2k-1} E_{2k-1} = 1, \text{ for } 
n \geq 1.
\end{equation}
\end{note}

\section{Some arithmetic properties of the sequences $a_{n}$ and $b_{n}$} 
\label{S:arithmetic} 

Given a sequence of integers $\{ x_{n} \}$ it is often interesting to 
examine its arithmetic properties. For instance, given a prime $p$, this is 
measured by the $p$-adic valuation $\nu_{p}(x_{n})$, defined as the largest 
power of $p$ that divides $x_{n}$. Examples of this process appear in 
\cite{amdeberhan-2008b} for the Stirling numbers and in 
\cite{amdeberhan-2008a, moll-2010a} for a sequence of coefficients arising 
from a definite integral.

The statements described below give information about $\nu_{p}(a_{n})$. These
results will be presented in a future
publication. 
M. Lasalle \cite{lasallem-2012a} established the next theorem by showing 
that $A_{n}$ and $C_{n}$ have the same parity. The fact that the Catalan 
numbers are odd if and only if $n = 2^{r}-1$ for some $r \geq 2$ provides 
the proof. This result appears in \cite{egecioglu-1983a, koshy-2006a}. 

\begin{theorem}
\label{lasalle-odd}
The integer $a_{n}$ is odd if and only if $n = 2(2^{m}-1)$.
\end{theorem}

The previous statement may be expressed in terms of the sequence of binary 
digits of $n$. 

\begin{fact}
Let $B(n)$ be the binary digits of $n$ and denote $\bar{x}$ a sequence of 
a arbitrary length consisting of the repetitions of the symbol $x$. The 
following statements hold (experimentally)

\smallskip

\noindent
1) $\nu_{2}(a_{n}) = 0$ if and only if $B(n) = \{ \bar{1},0 \}$.

\smallskip

\noindent
2) $\nu_{2}(a_{n}) = 1$ if and only if $B(n) = \{ \bar{1}\}$ or $\{ 1, \bar{0} 
\}$.

\smallskip

\noindent
3) $\nu_{2}(a_{n}) = 2$ if and only if $B(n) = \{ 1, 0, \bar{1}, 0\}$. 

\end{fact}

The experimental findings for the prime $p=3$ are described next.

\begin{fact}
Suppose $n$ is not of the form $3^{m}-1$. Then
\begin{equation}
\nu_{3}(a_{3n-2}) = \nu_{3}(a_{3n-1}) = \nu_{3}(a_{3n}).
\end{equation}
\noindent
Define $w_{j} = 3^{j}-1$.  Suppose 
$n$ lies in the interval $w_{j}+1 \leq n \leq w_{j+1}-1$. Then 
\begin{equation}
\nu_{3}(a_{3n+2}) = j - \nu_{3}(n+1).
\end{equation}
\noindent
If $n = w_{j}$, then $\nu_{3}(a_{3n}) = 0.$

Now assume that $n = 3^{m}-1$. Then 
\begin{equation}
\nu_{3}(a_{3n}) = \nu_{3}(a_{3n-1}) - 1 = \nu_{3}(a_{3n-2}) - 1 = m.
\end{equation}
\end{fact}

\begin{fact}
The last observation deals with the sequence $\{ a_{n}(\mu) \}$. 
Consider it now as defined by the recurrence \eqref{gen-recu}. The initial 
condition $a_{1}(\mu) = 2$, motivated by the origin of the sequence, in 
general does not provide integer entries.  For example, if $\mu = 2$, the 
sequence is 
\begin{equation*}
\left\{ 2, \, \frac{2}{3}, \, \frac{8}{9}, \, \frac{7}{3}, \, \frac{88}{9}, 
\, \frac{1594}{27}, \, \frac{1448}{3} \right\},
\end{equation*}
\noindent
and for $\mu =3$
\begin{equation*}
\left\{ 2, \, \frac{1}{2}, \, \frac{1}{2}, \, \frac{39}{40}, \, 3, 
\, \frac{263}{20}, \, \frac{309}{4} \right\}.
\end{equation*}
\noindent
Observe that the denominators of the sequence for $\mu = 2$ are always 
powers of $3$, but for $\mu = 3$ the arithmetic nature of the denominators 
is harder to predict. On the other hand if in the case $\mu = 3$ the 
initial condition is 
replaced by $a_{1}(3) = 4$, then the resulting sequence has denominators that 
are powers of $5$. This motivates the next definition.

\begin{definition}
Let $x_{n}$ be a sequence of rational numbers and $p$ be a prime. The 
sequence is called $p$-integral if the denominator of $x_{n}$ is a power of 
$p$. 
\end{definition}

Therefore if $a_{1}(3) = 4$, then the sequence $a_{n}(3)$ is $5$-integral. 
The same phenomena appeas for other values of $\mu$, the data is summarized 
in the next table.

\medskip

\begin{center}
\begin{tabular}{||c||c|c|c|c|c|c|c||}
\hline
$\mu$ & 2 & 3 &  4 & 5 & 6 & 7 & 8   \\
\hline 
$a_{1}(\mu)$ & 2 & 4 & 10 & 12 & 84 & 264 & 990 \\
\hline
$p$ & 3 & 5 & 7  & 7 & 11 & 11 & 13 \\
\hline
\end{tabular}
\end{center}

\begin{note}
The sequence $\{2, \, 4, \, 10, \, 12, \, 84, \, 264, \, 990 \}$ does 
not appear in Sloane's sequences list OEIS.
\end{note}

\medskip

This suggests the next conjecture.

\begin{conjecture}
Let $\mu \in \mathbb{N}$. Then there exists an initial condition 
$a_{1}(\mu)$ and a prime $p$ such that the sequence $a_{n}(\mu)$ 
is $p$-integral.
\end{conjecture}
\end{fact}

\medskip

Some elementary arithmetical properties of $a_{n}$ are discussed next. 
A classical result of E. Lucas states that a prime $p$ divides the binomial 
coefficient $\binom{a}{b}$ if and only if at least one of the base $p$ 
digits of $b$ is greater than the corresponding digit of $a$. 

\smallskip

\begin{prop}
\label{even-1}
Assume $n$ is odd. Then $a_{n}$ is even.
\end{prop}
\begin{proof}
Let $n = 2m+1$. The recurrence \eqref{rec-11} gives 
\begin{eqnarray*}
2(2m+1)a_{2m+1} & = & \sum_{k=1}^{2m} \binom{2m+1}{k-1} \binom{2m+1}{k+1} 
a_{k} a_{2m+1-k} \\
& = & 2 \sum_{k=1}^{m} \binom{2m+1}{k-1} \binom{2m+1}{k+1} 
a_{k} a_{2m+1-k}.
\end{eqnarray*}
\noindent
For $k$ in the range $1 \leq k \leq m$, one of the indices $k$ or $2m+1-k$ is 
odd. The induction argument shows that for each such $k$, either $a_{k}$ 
or $a_{2m+1-k}$ is an even 
integer. This completes the argument.
\end{proof}

\begin{lemma}
\label{lemma-odd1}
Assume $n = 2^{m}-1$. Then $\tfrac{1}{2}a_{n}$ is an odd integer.
\end{lemma}
\begin{proof}
Proposition \ref{even-1} shows that $\tfrac{1}{2}a_{n}$ is an integer. The 
relation \eqref{def-an} may be written as 
\begin{equation}
(-1)^{n-1}a_{n} = 2 + \frac{1}{n} \sum_{j=1}^{n-1} (-1)^{j} 
\binom{n}{j-1} \binom{n+1}{j+1} a_{j}.
\label{rel-newan}
\end{equation}
\noindent
This implies 
\begin{equation}
n \left[ (-1)^{n-1} \tfrac{1}{2}a_{n}- 1 \right] = 
\frac{1}{2} \sum_{j=1}^{n-1} (-1)^{j} \binom{n}{j-1} \binom{n+1}{j+1} a_{j}.
\end{equation}
\noindent
Observe that if $j$ is odd, then $a_{j}$ is even and $\binom{n+1}{j+1}$ 
is also even. Therefore the corresponding term in the sum is divisible by $4$.
If $j$ is even, then Lucas's theorem shows that 
$4$ divides $\binom{n+1}{j+1}$. It follows that the 
right hand side is an even number. This implies that $\tfrac{1}{2}a_{n}$ is 
odd, as claimed.
\end{proof}

The next statement, which provides the easier part of 
Theorem \ref{lasalle-odd}, describes the indices that 
produce odd values of $a_{n}$.

\begin{theorem}
If $n = 2(2^{m}-1)$, then $a_{n}$ is odd.
\end{theorem}
\begin{proof}
Isolate the term $j = n/2$ in the identity \eqref{rel-newan} to produce 
\begin{eqnarray*}
\left[ (-1)^{n} a_{n} + 2 \right] (2^{m}-1) & = & 
\binom{2^{m+1}-2}{2^{m}-2} \binom{2^{m+1}-1}{2^{m}} \, 
\tfrac{1}{2}a_{n/2} \\
& + & \frac{1}{2} \sum_{j \neq n/2} (-1)^{j} \binom{n}{j-1} \binom{n+1}{j+1} 
a_{j}. 
\end{eqnarray*}
\noindent
Lemma \ref{lemma-odd1} shows that $\frac{1}{2}a_{n/2}$ is odd and the 
binomial coefficients on the first term of the right-hand side are also odd
by Lucas' theorem. Each term of the sum is even because $a_{j}$ is even if 
$j$ is odd and for $j$ even $\binom{n}{j-1}$ is even. Therefore the entire 
right-hand side is even which forces $a_{n}$ to be odd.
\end{proof}

The final result discussed here deals with the parity of the sequence 
$b_{n}$. The main tool is the recurrence 
\begin{equation}
b_{n} = \sum_{k=1}^{n-1} \binom{n-1}{k} \binom{n-1}{k-1} b_{k}b_{n-k}
\end{equation}
\noindent
with $b_{1}=1$.  Observe that the binomial coefficients appearing in 
this recurrence are related to the Narayana numbers $N(n,k)$ 
\eqref{nara-numb-def}  by 
\begin{equation}
\binom{n-1}{k} \binom{n-1}{k-1} = (n-1)N(n-1,k-1).
\end{equation}
\noindent
Arithmetic properties of the Narayana numbers have been discussed by 
M. Bona and B. Sagan \cite{bona-2005a}. It is established that if $n = 
2^{m}-1$ then $N(n,k)$ is odd for $0 \leq k \leq n-1$; while if $n = 2^{m}$
then $N(n,k)$ is even for $1 \leq k \leq n-2$.  

\smallskip

The next theorem is the analog of M. Lasalle's result for the 
sequence $b_{n}$.

\begin{theorem}
The coefficient $b_{n}$ is an odd integer if and only if $n = 2^{m}$, for 
some $m \geq 0$.
\end{theorem}
\begin{proof}
The first few terms $b_{1}= 1, \, b_{2} = 1, 
\, b_{3} = 4$ support the base case of an inductive proof.

\smallskip

\noindent
If $n$ is odd, then 
\begin{equation}
b_{n} = (n-1) \sum_{k=1}^{n-1} N(n-1,k-1)b_{k}b_{n-k}
\label{nice-bn}
\end{equation}
\noindent
shows that $b_{n}$ is even. 

\smallskip

\noindent
Consider now the case $n = 2^{m}$. Then 
Lucas' theorem shows that $\binom{2^{m}-1}{k} 
\binom{2^{m}-1}{k-1}$ is odd for all $k$. The inductive step states that 
$b_{k}$ is even if $k \neq 2^{r}$. In the case $k = 2^{r}$, then $b_{n-k}$ is 
odd if and only if $k = 2^{m-1}$, in which case all the terms in \eqref{nice-bn}
are even with the single expection $\binom{2^{m}-1}{2^{m-1}} \binom{2^{m}-1}
{2^{m-1}-1}b_{2^{m-1}}^{2}$. This shows that $b_{n}$ is odd. 

\smallskip

\noindent
Finally, if $n = 2j$ is even with $j \neq 2^{r}$, then 
\begin{equation}
b_{n} = \binom{2j-1}{j} \binom{2j-1}{j-1}b_{j}^{2} + 
2 \sum_{k=1}^{j-1} \binom{n-1}{k} \binom{n-1}{k-1} b_{k}b_{n-k}.
\end{equation}
\noindent
Now simply observe that $j \neq 2^{r}$, therefore $b_{j}$ is even by 
induction. It follows 
that $b_{n}$ itself is even. 

\smallskip

This completes the proof.
\end{proof}

\section{One final question} 
\label{S:question} 

Sequences of combinatorial origin often turn out to be unimodal or logconcave.
Recall that a sequence $\{ x_{j}: 1 \leq j \leq n \}$ is 
called \textit{unimodal} if there 
is an index $m_{*}$ such that $x_{1} \leq x_{2} \leq \cdots \leq x_{m_{*}}$ and 
$x_{m_{*}+1} \geq x_{m_{*}+2} \geq \cdots \geq x_{n}$. The sequence is 
called \textit{logconcave} if $x_{n+1}x_{n-1} \geq x_{n}^{2}$. An 
elementary argument shows that a logconcave sequence is always unimodal. 
The reader will find in 
\cite{bona-2004b, boros-1999a, 
brenti-1994a,
butler-1987a, butler-1987b, 
choijy-2003a, 
sagan-1992a,stanley-1989a} a variety of examples of these 
type of sequences. 

\begin{conjecture}
The sequences $\{ a_{n} \}$ and $\{ b_{n} \}$ are logconcave.
\end{conjecture}

\noindent
\textbf{Acknowledgements}. The work of the second author was 
partially supported by NSF-DMS 0070567.


\end{document}